\documentclass[12pt,a4paper]{article}
\usepackage{amssymb,comment,amsmath,amsthm}
\usepackage{graphicx}
\usepackage{color}
\usepackage{titling,titlesec}
\usepackage{url}
\usepackage{lipsum}
\titleformat{\subsection}[runin]
            {\normalfont\bfseries}{}{0em}{}[.]

\newfont{\blb}{msbm10 scaled\magstep1}
\newfont{\comp}{cmr12 scaled\magstep1}
\newfont{\compb}{cmr10 scaled\magstep2}
\newfont{\sbb}{cmssbx10 scaled\magstep3}
\newfont{\sbbb}{cmssbx10 scaled\magstep5}
\newfont{\sbs}{cmssbx10 scaled\magstep1}
\newtheorem{theorem}{Theorem}
\newtheorem{lemma}{Lemma}

\newtheorem{problem}{Problem}
\newtheorem{proposition}{Proposition}
\newtheorem{aproposition}{Proposition}

\newcommand{\dm}[1]{\textcolor{blue}{\textbf{[DM:} #1\textbf{]}}}

\newcommand{\PP}{{\mathbb P}}
\newcommand{\EE}{{\mathbb{E}}}

\title{Erd\H{o}s-Rogers functions for arbitrary pairs of graphs}

\author{
Dhruv Mubayi\thanks{Department of Mathematics, Statistics and Computer Science, University of Illinois, Chicago, IL 60607. Email: mubayi@uic.edu. Research partially supported by NSF Awards
DMS-1952767 and DMS-2153576 and a Simons Fellowship.} \and
Jacques Verstra\" ete\thanks{Department of Mathematics, University of California, San Diego, CA, 92093-0112 USA.
			Email: jverstraete@ucsd.edu.
			Research supported by NSF award DMS-1952786.  }}

\begin{document}

\date{ }
\setlength{\droptitle}{-3cm}

\maketitle

\begin{abstract}
Let $f_{F,G}(n)$ be the largest size of an induced $F$-free subgraph that every $n$-vertex $G$-free graph is guaranteed to contain. We prove that for any triangle-free graph $F$,
\[ f_{F,K_3}(n) = f_{K_2,K_3}(n)^{1 + o(1)} = n^{\frac{1}{2} + o(1)}.\]
Along the way we give a slight improvement  of a construction of Erd\H os-Frankl-R\"odl for the Brown-Erd\H os-S\'os $(3r-3,3)$-problem  when  $r$ is large.

In contrast to our result for $K_3$, for  any $K_4$-free graph $F$ containing a cycle, we prove there exists $c_F > 0$ such that $$f_{F,K_4}(n) > f_{K_2,K_4}(n)^{1 + c_F} = n^{\frac{1}{3}+c_F+o(1)}.$$
\iffalse We also observe that our earlier proof for $F=K_3$ generalizes to  $f_{F,K_4}(n) = O(\sqrt{n}\log n)$ for all  $F$ containing a cycle.  \fi

For every graph  $G$, we prove that there exists $\varepsilon_G >0$ such that whenever $F$ is a non-empty graph such that $G$ is not contained in any blowup of $F$, then $f_{F,G}(n) = O(n^{1-\varepsilon_G})$. On the other hand, for  graph $G$ that is not a clique, and every $\varepsilon>0$, we exhibit a  $G$-free graph $F$ such that $f_{F,G}(n) = \Omega(n^{1-\varepsilon})$.
\end{abstract}

\section{Introduction}
Say that a graph is $F$-free if it contains no subgraph isomorphic to $F$. Denote by $f_{F,G}(n)$ the maximum $m$ such that every $n$-vertex $G$-free graph contains an induced $F$-free subgraph on at least $m$ vertices.
Hence the assertion $f_{F,G}(n)<b$ means that there exists an $n$ vertex $G$-free graph $H$ such that every vertex subset of $H$ of size $b$ contains a copy of $F$. 
The case $F = K_s$ and $G = K_t$
is the Erd\H{o}s-Rogers~\cite{ER} function $f_{s,t}(n)$. Classical results in Ramsey Theory~\cite{AKS, K} give $r(3,t) = \Theta(t^2/\log t)$, which shows $f_{K_2,K_3}(n) = \Theta(\sqrt{n\log n})$.
We prove that roughly the same holds for $f_{F,K_3}(n)$ for any triangle-free graph $F$:

\begin{theorem}\label{triangles}
For any triangle-free graph $F$ containing at least one edge,
\[ f_{F,K_3}(n) = n^{\frac{1}{2} + O(\sqrt{\frac{\log\log n}{\log n}})} = f_{K_2,K_3}(n)^{1 + o(1)}.\]
\end{theorem}
Our bound in Theorem~\ref{triangles} is much larger than $f_{K_2,K_3}(n) = \Theta(\sqrt{n\log n})$, and therefore the following problem seems natural.

\begin{problem} \label{probFK3}Find a triangle free $F$ for which $f_{F, K_3}(n)/f_{K_2, K_3}(n) \rightarrow \infty$.
\end{problem}

A large  pseudorandom triangle free graph with many edges seems  an obvious choice for $F$ in Problem~\ref{probFK3}. Perhaps the simpler $F=K_{t,t}$ is another example. More generally, for each $s \ge 3$, one can ask whether there exists a $K_s$-free $F$ for which $f_{F, K_s}(n)/f_{K_{s-1}, K_s}(n) \rightarrow \infty$.

Unlike the case of triangles, it appears that for $s \geq 4$, it is difficult to determine for each $K_s$-free graph $F$ a constant $c = c(F)$ such that $f_{F,K_s}(n) = n^{c + o(1)}$. The second author and Mattheus~\cite{MaV} proved $f_{K_2,K_4}(n) = n^{1/3 + o(1)}$ whereas it is well-known that $f_{K_3,K_4} = n^{1/2 + o(1)}$.
We~\cite{MVER1} recently proved $f_{K_3, K_4}(n) = O(\sqrt n \log n)$ and the proof can be extended to prove that for every $K_4$-free graph $F$, we have
$f_{F, K_4}(n) = O(\sqrt n \log n)$.  Perhaps this can be improved for triangle-free $F$ as follows.

\begin{problem} \label{probK3FK4} Is it true that for every triangle-free graph $F$ there exists $\varepsilon=\varepsilon_F>0$ such that
$f_{F, K_4}(n)< n^{1/2-\varepsilon}$?
\end{problem}
Regardless of whether Problem~\ref{probK3FK4} has an affirmative answer, one might suspect that there exists a sequence of triangle-free graphs where the exponent tends to $1/2$. We propose the following.

\begin{problem} \label{probKttK4}
Prove (or disprove) that $f_{K_{t,t}, K_4}(n) = n^{1/2+o_t(1)}$. \end{problem}
The method of proof of Theorem~\ref{c4} yields  $f_{K_{t,t}, K_4}(n) > n^{2/5-o_t(1)}$.

Our next result shows that for $s \ge 4$, we can find substantially larger $F$-free sets in $K_s$-free graphs than their conjectured~\cite{MV}  minimum independence number, which is  $n^{1/(s-1)+o(1)}$.

\begin{theorem}\label{c4k4}
Let $s \geq 4$ and let $F$ be any graph containing a cycle. Then there exists a constant $c_F > 0$ such that
\[ f_{F,K_s}(n) = \Omega(n^{\frac{1}{s - 1} + c_F}).\]
\end{theorem}

If $F$ is a cycle, then this bound is almost tight for $K_4$, using the following proposition. Write $r(H, t)$ for the ramsey number of $H$ versus a clique on $t$ vertices.

\begin{proposition} \label{prop} For any graphs $F$ and $G$, 
    $$f_{F, G}(r(G, t)-1) < r(F, t).$$
\end{proposition}
Indeed,  let $H$ be a $G$-free  graph on $r(G, t)-1$ vertices with no independent set of size $t$. Then the maximum $F$-free subset of $H$ has size less than $m:=r(F, t)$ for any set of $m$ vertices in $H$ must contain either a copy of $F$ or an independent set of size $t$.

When $F=C_{2k}$ or $F=C_{2k-1}$ we have $r(F, t) = O(t^{k/(k-1)}/(\log t)^{1/(k-1)})$ (\cite{LZ, Sudakov}). Moreover, 
 recent results of~\cite{MaV}, yield $r(K_4, t) = \Omega(t^3/\log^4 t)$. Putting these together in Proposition~\ref{prop} yields
 \begin{equation} \label{eqn:upper}f_{F, K_4}(n) = O(n^{\frac{k}{3k-3}} (\log n)^{\frac{4k-3}{3k-3}}) \qquad \hbox{ for } \qquad \hbox{ $F\in\{C_{2k}, C_{2k-1}\}$}.\end{equation}
The constant in Theorem~\ref{c4k4} satisfies $c_{F}= \Theta(1/k)$ for $F=C_k$ and with (\ref{eqn:upper}) this gives
\begin{equation}\label{eqn:kcycle} f_{C_k,K_4}(n) = f_{K_2,K_4}(n)^{1 + \Theta(\frac{1}{k})+o(1)} = n^{\frac{1}{3} + \Theta(\frac{1}{k})+o(1)}.
\end{equation}

This shows that there are graphs $F$ for which $f_{F,K_4}(n)$ does not have the same exponent as $f_{K_2,K_4}(n)$ or $f_{K_3,K_4}(n)$, in contrast to the case of $f_{F,K_3}(n) = f_{K_2,K_3}(n)^{1 + o(1)}$ from Theorem \ref{triangles}.
Using the graphs constructed in~\cite{MaV}, and following the analysis along the lines 
of Janzer and Sudakov~\cite{JS}, Balogh et al.~\cite{Balogh} improved the upper bound in (\ref{eqn:upper}) slightly in the case of even cycles, by showing
\[ f_{C_{2k},K_4}(n) = O(n^{\frac{k}{3k-2}}(\log n)^{\frac{6k}{3k-2}}).\]
They also showed for complete multipartite graphs 
\[ f_{K_{s_1,...,s_r},K_{r + 2}}(n)= O(n^{\frac{2s-3}{4s-5}}(\log n)^3),\]
where $s=\sum s_i$. In the special case of 4-cycles this gives $f_{C_4,K_4}(n) = O(n^{5/11})$. 

We now address general Erd\H{o}s-Rogers functions $f_{F,G}(n)$. For a given $G$, the first natural question  is when $f_{F,G}(n)$ can be $n^{1 - o(1)}$ as $|V(F)| \rightarrow \infty$.
A {\em blowup} of a graph $F$ is obtained by replacing each vertex $v$ of $F$ with an independent set $I_v$ and
adding all edges between $I_u$ and $I_v$ whenever $\{u,v\} \in E(F)$. The graph $F$ is a homomorphic image of $G$ if and only if some blowup of $F$ contains $G$.
Consequently, we say that $F$ is {\em $\mbox{hom}(G)$-free} if no blowup of $F$ contains $G$.
For instance, if $G$ is bipartite and $F$ contains at least one edge, then blowups of $F$ contain arbitrarily large complete bipartite graphs, and therefore
$F$ is not $\mbox{hom}(G)$-free. This condition turns out to determine when Erd\H{o}s-Rogers functions $f_{F,G}(n)$ can approach $n^{1 - o(1)}$ as
$|V(F)| \rightarrow \infty$:

\begin{theorem}\label{general}
For every graph $G$, there exists $\varepsilon_G >0$ such that if $F$ is any $\mbox{hom}(G)$-free graph containing at least one edge, then
\[ f_{F,G}(n)  = O(n^{1-\varepsilon_G}).\]
On the other hand,  if $G$ is not a clique, then for any $\varepsilon >0$ there exists a $G$-free graph $F$ such that $f_{F,G}(n)  =\Omega(n^{1- \varepsilon})$.
\end{theorem}

If $G$ is a clique, then every $G$-free graph is also $\mbox{hom}(G)$-free, hence the first part of Theorem~\ref{general} applies to all $G$-free graphs $F$ when $G$ is a clique. As mentioned earlier, in the case $G = K_4$, it turns out $f_{F,G}(n) = O(n^{1/2}\cdot \log n)$ due to our results in~\cite{MVER1},
so we may take $\varepsilon_{K_4} \geq 1/2$. It appears to be difficult to determine the largest possible value of $\varepsilon_G$ for each
graph $G$ in Theorem \ref{general}.

\iffalse
{\bf Is it true that $b_G \leq 1/2$ for many $G$? 
Balogh et al. prove the very nice result that $b_{K_s} < 1/2$ when $s \geq 4$ and $F$ is complete $(s - 2)$-partite. We know 
$b_{K_s} \leq 1/2$ for all $(s - 1)$-partite $F$ by the results of~\cite{MV2}. What explicit examples have 
$b_G$ close to $1$ with $F$ is $\mbox{hom}(G)$-free? Would seem to be easily true if $G$ is a long odd cycle and $F$ is 
$C_4$, using random graphs, or any sparse $G$ compared to a dense enough $F$. } 
\fi

\section{Proof of Theorem \ref{triangles}}

Ajtai, Koml\'{o}s and Szemer\'{e}di~\cite{AKS} and Shearer~\cite{Sh} proved that $r(3,t) = O(t^2/\log t)$. Using the random triangle-free
process, Kim~\cite{K} (see also Fiz Pontiveros, Griffths and Morris~\cite{FGM} and Bohman and Keevash~\cite{BK0}) showed
$r(3,t) = \Omega(t^2/\log t)$, thereby determining the order of magnitude of $r(3,t)$. Consequently, for any non-empty graph $F$,
\[ f_{F,K_3}(n)  \geq f_{K_2,K_3}(n) = \Theta(\sqrt{n\log n}).\]
To prove Theorem \ref{triangles} we employ a construction of Erd\H{o}s, Frankl and R\"{o}dl~\cite{EFR} of a linear triangle-free $R$-uniform $N$-vertex hypergraph. In the appendix, we give present a minor modification of their construction which gives a bound that is better than the bound from~\cite{EFR} when $R > \log N$; they prove a lower bound  $N^2/e^{O(\log R \sqrt{\log N})}$ while our bound is $N^2/e^{O( \sqrt{\log R\log N})}$.

\begin{theorem} {\bf (Proposition A in Appendix)} \label{thmEFR}
For any $R, N \ge 3$ and $N\ge R \geq \log N$, there exists an $N$-vertex $R$-uniform hypergraph $H$ with the following properties:
\begin{center}
\begin{tabular}{lp{5in}}
$\mathrm{(i)}$ & $|E(H)| \geq N^2/R^{8\sqrt{\log_R N}}$ \\
$\mathrm{(ii)}$ & $H$ is linear, that is, for any distinct edges $e,f \in H$, $|e \cap f| \leq 1$. \\
$\mathrm{(iii)}$ & $H$ is triangle-free, that is, for any three distinct edges $e,f,g \in H$, if $|e \cap f| = |f \cap g| = |g \cap e| = 1$ then $|e \cap f \cap g| = 1$.
\end{tabular}
\end{center}
\end{theorem}

{\bf Proof of Theorem~\ref{triangles}.}
We are to prove that
\[ f_{F,K_3}(n) = n^{\frac{1}{2} + O(\sqrt{\frac{\log\log n}{\log n}})}.\]
Let $t = |V(F)|$. 
We apply Theorem~\ref{thmEFR} with $R = \lceil 3t\log t \log N\rceil$, where $t = |V(F)|$. Then (i) yields
\begin{equation}\label{eh}
 |E(H)| \geq \frac{N^2}{R^{8\sqrt{\log_R N}}} = N^{2 - O(\sqrt{\frac{\log\log N}{\log N}})}.
 \end{equation}
Let $G$ be the graph whose vertex set is $E(H)$ and where $E(G) = \{e,f \in E(H) : e \cap f \neq \emptyset\}$.
For each vertex $v \in V(H)$, the set $K_v = \{e \in E(H) : v \in e\}$ induces a clique in $G$.
If for some distinct $v,w \in V(H)$ there exist distinct $e,f \in K_v \cap K_w$, then by definition $v,w \in e \cap f$,
which contradicts that $H$ is linear. Therefore $|V(K_v) \cap |V(K_w)| \leq 1$ for all distinct $v,w \in V(H)$, and the cliques
$K_v$ are edge-disjoint in $G$. Similarly, since $H$ is triangle-free, every triangle in $G$ is contained in a clique $K_v$ for some $v \in V(H)$.

Independently for $v \in V(H)$, let $\chi_v : V(K_v) \rightarrow V(F)$ be a random coloring of $K_v$. Next, we remove all edges $\{x,y\}$ of $G[K_v]$
such that $\chi_v(x) = \chi_v(y)$ or $\chi_v(x)\chi_v(y) \not\in E(F)$. In other words, we have placed a blowup of a copy of $F$ in each set $K_v$. 

Since $F$ contains no triangle, the resulting graph $G^*$ is triangle-free. We now prove that $G^*$ has no $F$-free induced subgraph
with at least $N$ vertices. To see this, fix a set $Z$ of $N$ vertices of $G^*$. The probability that $Z$ is an $F$-free  set of $G^*$ is 
\[ \PP(Z) \le \prod_{v \in V(H)} t \cdot \Bigl(1 - \frac{1}{t}\Bigr)^{|K_v \cap Z|}.\]
Since $H$ is $R$-uniform,
\[ \sum_{v \in V(H)} |K_v \cap Z| = \sum_{e \in Z} |e| = R|Z| = RN.\]
Using $(1 - x)^y \leq e^{-xy}$ for $0 \leq x \leq 1$ and $y \geq 1$,
\[ \PP(Z) \le t^N \Bigl(1 - \frac{1}{t}\Bigr)^{RN} \leq e^{N\log t - RN/t} < N^{-2N}.\]
The number of sets of size $N$ in $G^*$ is no more than
\[ {N^2 \choose N} \leq \frac{N^{2N}}{N!}.\]
Therefore the expected number of $F$-free sets $Z$ of size $N$ in $G^*$ is less than $1/N!$. We may therefore select $G^*$ so as to contain no $F$-free
subgraph with at least $N$ vertices. Since $G^*$ is triangle-free, and $n := |V(G^*)| = |E(H)|$, the bound (\ref{eh}) gives
\[ f_{F,K_3}(n) < N = n^{\frac{1}{2} + O(\sqrt{\frac{\log\log n}{\log n}})}.\]
This proves the theorem. \qed

\section{Proof of Theorem \ref{c4k4}: Large $C_k$-free subsets}

To prove Theorem \ref{c4k4}, it is sufficient to prove the following theorem:

\begin{theorem}\label{c4}
For any graph $F$ containing a cycle $C_k$, there exists $\epsilon_k > 1/100k$ such that
\[ f_{F,K_4}(n) = \Omega(n^{\frac{1}{3} + \epsilon_k}).\]
\end{theorem}

To see that this implies Theorem \ref{c4k4}, let $H$ be a $K_s$-free graph where $s \geq 5$. If $H$ has maximum degree $d$,
then by Tur\'{a}n's Theorem, $H$ has an independent set of size at least $n/(d + 1)$, and the neighborhood of a vertex of degree
$d$ induces a $K_{s - 1}$-free subgraph. By induction, setting $\alpha_k(4) = 1/3 + \epsilon_k$, for $s \geq 5$, there exists $\alpha = \alpha_k(s - 1) > 1/(s - 2)$
such that this $K_{s - 1}$-free subgraph has an $F$-free subgraph with $\Omega(d^{\alpha_k(s - 1)})$
vertices. Therefore we have an $F$-free subgraph of size at least
\[ \Omega(\max\Bigl\{d^{\alpha_k(s - 1)},\frac{n}{d + 1}\Bigr\}).\]
Setting
\[ \alpha_k(s) = 1 - \frac{1}{1 + \alpha_k(s - 1)},\]
since $\alpha_k(4) > 1/3$ for all $k \geq 3$, by induction we have 
\[ \alpha_k(s) > 1 - \frac{1}{1 + \frac{1}{s-2}} = \frac{1}{s-1}\]
as required. Moreover, if $\alpha_k(s-1) \ge 1/(s-2)+\epsilon$ where $\epsilon\le 1$, the calculation above yields
\begin{align*}
\alpha_k(s) \ge    1-\frac{s-2}{s-1+\epsilon(s-2)}
 &=\frac{1}{s-1}+\epsilon\left(\frac{(s-3)+\frac{1}{s-1}}{s-1+\epsilon(s-2)}\right) \\ 
&>\frac{1}{s-1}+\epsilon\left(\frac{s-3}{2(s-1)}\right)\\
& \ge \frac{1}{s-1}+\frac{\epsilon}{4} .
\end{align*}
 
With $\alpha_k(4) > 1/3 + 1/100k$, this gives 
$\alpha_k(s) =  1/(s - 1) + \Omega_s(1/k)$ as $k \rightarrow \infty$.

\bigskip
We will prove Theorem~\ref{c4} as follows: a given $K_4$-free graph $H$ either has few $k$-cycles going through every vertex or has a vertex that lies in many $k$-cycles. In the former case, we apply standard results about hypergraph independent sets (Lemma~\ref{spencer}) to obtain a large $C_k$-free subset. In the latter case, we show that $H$ contains a dense bipartite graph and then use the dependent random choice technique to extract from this a large independent set in one of the parts. These assertions are stated in the next three lemmas.

For sets $X,Y$ of vertices in a graph $G$, let $e(X,Y)$ denote the number of edges $\{x,y\} \in E(G)$
such that $x \in X$ and $y \in Y$.

\begin{lemma}\label{ckprop}
Let $G$ be a graph of maximum degree $d$, and let $\delta > 0$. Suppose the number of cycles of length $k$
containing a vertex $v_0 \in V(G)$ is at least $\delta d^{k-1}$. Then there exist sets $X,Y \subseteq V(G)$
such that $e(X,Y) \geq \delta |X||Y|/(2\log_2 d)^k$ and $|X|,|Y| \geq \delta d/(\log_2 d)^{k-3}$.
\end{lemma}

\begin{proof}
Let $\mathcal{C}$ be the set of $k$-cycles containing $v_0$. For each $\sigma \in \mathcal{C}$, pick an ordering 
$(\sigma_0,\sigma_1,\dots,\sigma_{k-1},\sigma_0)$ of the vertices of $\sigma$, where $\{\sigma_i,\sigma_{i + 1}\} \in E(\sigma)$ with subscripts modulo $k$. Let $X_i = \{\sigma_i  : \sigma \in \mathcal{C}\}$ for $1 \leq i \leq k - 1$. Then for $2 \leq i \leq k - 2$ there exist sets $X_i' \subseteq X_i$
and $a_i \in \{1,2,\dots,d\}$ such that every vertex of $X_i'$ has at least $a_i/2$ and at most $a_i$ neighbors in $X_{i-1}'$, and the number of cycles $\sigma \in \mathcal{C}$ with $\sigma_i \in X_i'$ is at least $\delta d^{k - 1}/(\log_2 d)^{k-3}$. This can be done
iteratively, starting by splitting $X_2$ into sets $X_{2j}$ such that every vertex of $X_{2j}$ has at least $d/2^{j + 1}$ and at most $d/2^{j}$ neighbors
in $X_1$, for $0 \leq j \leq \log_2 d$, and considering an $X_2' = X_{2j}$ for which at least $|\mathcal{C}|/(\log_2 d)$ of the cycles use an edge between
$X_1$ and $X_{2}'$. Call this collection of cycles $\mathcal{C}_2$. Then repeat the argument for the pair $X_{2}'$ and $X_3$, with collection of cycles $\mathcal{C}_2$ 
so there exists $X_{3}' \subseteq X_3$ and  $\mathcal{C}_3 \subset \mathcal{C}_2$ with
$|\mathcal{C}_3| \ge |\mathcal{C}_2|/\log_2 d$.
We continue to obtain $X_i' \subseteq X_i$ for all $i \leq k - 2$ and set $\mathcal{C}':= \mathcal{C}_{k-2}$. Then 
\[ |\mathcal{C}'| \geq \frac{|\mathcal{C}|}{(\log_2 d)^{k-3}} \ge \frac{\delta d^{k-1}}{(\log_2 d)^{k-3}}\] 
and for every 
$\sigma \in \mathcal{C}'$ we have  $\sigma_i \in X_i'$ for $i \leq k - 2$. 
\medskip

Let $X = X_{k-2}'$ and $Y = X_{k-1}'$.
The number of cycles in $\mathcal{C}'$ containing an edge $\{x,y\}$ with $x \in X$ and $y \in Y$ is at most $a_2 \cdots a_{k-2}\le d^{k-3}$ as the maximum degree is $d$. Consequently,
$$|\mathcal{C}'| \le e(X, Y) \cdot a_2\cdots a_{k-2} \le e(X, Y) \cdot d^{k-3}$$
and
\[ d \cdot \min\{|X|,|Y|\} \geq e(X,Y) \geq \frac{|\mathcal{C}'|}{d^{k-3}} \geq \frac{\delta d^2}{(\log_2 d)^{k-3}}.\]
Therefore $\min\{|X|,|Y|\} \geq \delta d/(\log_2 d)^{k-3}$.
\medskip

Next we prove that $e(X,Y) \geq \delta |X||Y|/(2\log_2 d)^k$.
By construction, for $2 \leq i \leq k - 2$,
\[ \frac{a_i}{2} |X_i'| \leq e(X_i',X_{i-1}') \leq d|X_{i-1}'|\]
and therefore $a_i \leq 2d|X_{i-1}'|/|X_i'|$. Since $|X_1'|\le |N(v_0)| \le d$ and $|Y|\le |N(v_0)| \le  d$, 
\[ a_2 a_3 \cdots a_{k - 2} \leq 
\prod_{i=2}^{k-2} 2d\frac{|X_{i-1}'|}{|X_i'|} = (2d)^{k-3}\frac{|X_1'|}{|X_{k-2}'|}=
(2d)^{k-3}\frac{|X_1'|}{|X|} \le \frac{(2d)^{k-1}}{|X||Y|}.\]
Consequently, 
\[ e(X,Y) \geq \frac{|\mathcal{C}'|}{a_2 \dots a_{k-2}} \geq \frac{|\mathcal{C}'||X||Y|}{(2d)^{k-1}} \ge \frac{\delta}{2^{k-1}(\log_2 d)^{k-3}}|X||Y|\]
completing the proof.
\end{proof}

The following lemma is a standard consequence of the dependent random choice method and we omit the  proof.
\begin{lemma}\label{drc}
Let $\gamma \geq 0$, $s \geq 1$, and let $X$ and $Y$ be disjoint sets of vertices in a graph, such that $e(X,Y) \geq \gamma |X||Y|$. Then for any 
$s \geq 1$, there exists a set $Z \subseteq Y$ such that 
\[ |Z| \geq \frac{1}{2}\gamma^s |Y|\]
and every pair of vertices in $Z$ has at least $\gamma |X||Y|^{-1/s}$ neighbors in $X$. 
\end{lemma}
Finally we need the following standard result about independent sets in hypergraphs first proved by Spencer~\cite{Spencer}.
\begin{lemma} \label{spencer}
For every $k \ge 2$, every $n$-vertex $k$-uniform hypergraph with average degree $d>0$ has an independent set of size at least $(1-1/k) n/ d^{1/(k-1)}$.
\end{lemma}

We now have the necessary ingredients to prove Theorem~\ref{c4}.
\bigskip

{\bf Proof of Theorem \ref{c4}.}  For $k \geq 3$, let 
\[ \epsilon_k = \frac{1}{100(k-1)}.\]
 Let $H$ be an $n$-vertex $K_4$-free graph with maximum degree $d$.  We will find a $C_k$-free subset of vertices in $H$ of size at least $n^{1/3+\epsilon_k}$. Suppose that $\triangle$ is the maximum number of copies of $C_k$ that a vertex is in. Define 
$$\delta:=\frac{\triangle}{ d^{k-1}}.$$

We now obtain two different bounds on the maximum $C_k$-free set.
\medskip

{\bf Bound 1.}  Let $\mathcal{H}$ be the $k$-uniform hypergraph with $V(\mathcal{H}) = V(H)$ and $E(\mathcal{H}) = \{V(C_k) : C_k\subseteq H\}$. Then $\mathcal{H}$ has maximum degree (and hence average degree)  at most $\triangle$ and Lemma~\ref{spencer} implies that $H$ has an independent set of size at least
$$\Omega\left( 
\frac{n}
{\triangle^{
\frac{1}{k-1}
}}
\right) =  \Omega\left( \frac{n}{\delta^{\frac{1}{k-1}} d}
\right).$$

{\bf Bound 2.} Let $v_0 \in V(H)$ lie in $\triangle=\delta d^{k - 1}$ copies of $C_k$.  
By Lemma~\ref{ckprop}, there exist sets $X,Y \subseteq V(H)$ such that 
\[ e(X,Y) \geq \frac{\delta}{(2\log_2 d)^k} |X||Y| =: \gamma |X||Y|,\]
where $|X| \geq |Y| \geq \gamma d$. 
By Lemma~\ref{drc} applied with $s=3$, there exists $Z \subseteq Y$ with
\[ |Z| \geq \frac{1}{2}\gamma^3 |Y|\]
such that every pair of vertices in $Z$ has at least $\gamma |X| |Y|^{-1/3} \geq \gamma |Y|^{2/3}$ common neighbors in $X$. If $Z$ is not an independent set in $H$, then there exists $\{x,y\} \in E(H)$ with $x,y \in Z$. Since $H$ is $K_4$-free, $N(x) \cap N(y)$ is an independent set in $H$ of size at least  $\gamma |Y|^{2/3} \geq \gamma^{5/3} d^{2/3}$. Otherwise, $Z$ is an independent set in $H$ of size at least $\frac{1}{2}\gamma^3 |Y| \geq \frac{1}{2}\gamma^{4}d$. In particular, $H$ has a 
$C_k$-free induced subgraph of size at least 
$$h(d,\gamma) =  \min \left\{\gamma^{5/3} d^{2/3}, \frac{1}{2}\gamma^{4} d
\right\}.$$

It is also the case that $G$ always contains an independent set with at least $n/(d + 1)$ vertices, by Tur\'{a}n's Theorem.
If $d \leq n^{2/3 - \epsilon_k}$, this gives an independent set of size $n^{1/3 + \epsilon_k}$ in $G$. If $d \geq n^{2/3 + 2\epsilon_k}$, then since the neighborhood of a vertex of degree $d$ induces a triangle-free graph, this neighborhood contains an independent set of size at least $d^{1/2} \ge n^{1/3 + \epsilon_k}$ in $G$. Therefore we assume $n^{2/3 - \epsilon_k} \leq d \leq n^{2/3 + 2\epsilon_k}$.
In that case, by Bounds 1 and 2,  we obtain a $C_k$-free set of size at least
$$\Omega\left(
\max\left\{\frac{n}{\delta^{\frac{1}{k-1}} d} , h(d, \gamma)\right\}\right).$$
If $\delta < n^{-1/25}$, then Bound 1 is at least 
$$\Omega\left( \frac{n}{\delta^{\frac{1}{k-1}} d}
\right) = \Omega\left( n^{\frac{1}{3} + \frac{1}{25(k-1)} - 2\epsilon_k}
\right) = \Omega(n^{\frac{1}{3} + \epsilon_k})$$
as $\epsilon_k < 1/75(k-1)$.
So we may assume that $\delta \ge n^{-1/25}$, and, as $n$ is sufficiently large, we may assume that $\gamma =\delta/(2\log_2 d)^k > n^{-1/24}$. In  this case, $\epsilon_k<1/100$ and $d>n^{2/3-\epsilon_k}$ yield
$$\gamma^{\frac{5}{3}}d^{\frac{2}{3}} >n^{\frac{-5}{72} + \frac{4}{9} - \frac{2\epsilon_k}{3}} > n^{\frac{1}{3} + \epsilon_k} \qquad \hbox{ and } \qquad \gamma^4d> n^{-\frac{1}{6}+\frac{2}{3} -\epsilon_k}>
2 n^{\frac{1}{3} + \epsilon_k}$$
and therefore $h(d,\gamma)>n^{1/3+\epsilon_k}$, completing the proof. \qed

\section{Proof of Theorem~\ref{general}}

A sunflower is a collection of sets every pair of which have the same intersection, called the core. We need the well-known Erdos-Rado sunflower lemma in the form below. 

\begin{lemma} \label{ersun}
Fix $t,m>0$. Every $t$-uniform hypergraph with more than $t!(m-1)^t$ edges has a sunflower of size $m$.
 \end{lemma}

{\bf Proof of Theorem~\ref{general}.} Let $|V(G)| = k$. We may assume that $G$ is not acyclic, since otherwise $G$ would be contained in a blowup of $F$. Consider an $n$ by $n$ bipartite graph $H$ without cycles of length at most $2k$ and where every
vertex has degree $d = n^{\frac{1}{3k}}$. Such graphs exist, for example the bipartite Ramanujan graphs of
Lubotzky, Phillips and Sarnak~\cite{LPS}, or even a random $d$-regular graph (if we are not fussy about the constant in the exponent). We now employ the methods
of~\cite{MV,CMMV}. Let $H'$ be the restriction
of the square of $H$ to one part of $H$, so that $H'$ has $n$ vertices, and is a union of
$n$ edge-disjoint cliques $K^1,K^2,\dots,K^n$ of order $d$. Since $G$ is not acyclic,   every copy of $G$ in $H'$ is contained in one of those cliques.
In each of the cliques, take a random coloring with elements of $V(F)$, and put an edge
between any two color classes corresponding to an edge of $F$. Since $F$ is $\mbox{hom}(G)$-free,
this random graph $H^*$ is $G$-free. We claim (similarly to the proof of Theorem \ref{triangles}), that every set of at least $(2n|V(F)|\log |V(F)|)/d$ vertices of $H^*$
induces a copy of $F$. The probability that such a set $X$ does not induce a copy of $F$ is at most 
\[ \prod_{i = 1}^n |V(F)| \cdot \Bigl(1 - \frac{1}{|V(F)|}\Bigr)^{|X \cap V(K^i)|}.\]
Now we use
\[ \sum_{i = 1}^n |X \cap V(K_i)| = d|X|\]
and therefore the expected number of such $X$ is at most
\[ {n \choose |X|} \cdot |V(F)|^n \Bigl(1 - \frac{1}{|V(F)|}\Bigr)^{d|X|} < e^{|X|\log n - d|X|/|V(F)| + n\log |V(F)|}.\]
This is vanishing since $d|X|/|V(F)| > 2n\log |V(F)|$. Therefore
\[ f_{F,G}(n) = O(n/d) = O(n^{1-\frac{1}{3k}}) = O(n^{1 - \frac{1}{3|V(G)|}})\]
and we may take $\varepsilon_G = 1/3|V(G)|$ in Theorem \ref{general}.

\bigskip

We now prove  the second statement of the theorem.  Let $r:= |V(G)|-1$. If $G$ is acyclic, then any $n$-vertex $G$-free graph has an independent set $I$ of size linear in $n$, and $I$ is certainly $F$-free for any nonempty $F$ so we are done.
If $G$ is not 2-connected, then let $F=K_r$ so that $F$ is clearly $G$-free. Suppose that $H$ is an $n$-vertex  $G$-free graph. Then no two  $r$-cliques  in $H$ have a point in common, for otherwise the subgraph of $H$ induced by their union contains $G$. Indeed, we can pick some vertex in the intersection of the two cliques to be a cut vertex of $G$, and then easily embed $G$ in the union of the two cliques (the embedding is even easier if $G$ is not connected). Consequently, the $r$-cliques in $H$ are pairwise vertex disjoint. Then $H$ has a $K_r$-free induced subgraph of size at least $(1-1/r)n$ and we are done.

We may henceforth assume that $G$ is 2-connected.   Since $G$ is not a clique, let $v,w$ be nonadjacent vertices in $G$.
Let $G^+$ be the graph obtained from $G$ by adding all edges that are not already in $G$ between $\{v,w\}$ and $N_G(v) \cup N_G(w)$.  So $G^+ \supset G$, and $v$ and $w$ are clones in $G^+$. Let $G^{*}  =G^+ -\{w\}$  and 
let $G^{**} =G^+ -\{v,w\} = G^*-\{v\}$. So $G^*$ has $r$ vertices and  $G^{**}$ has $r-1 $ vertices.

Assume that $t$ is sufficiently large in terms of $r$ and set $\delta = 1/5r^2$. Apply Proposition~\ref{random} to obtain a $t$-vertex $r$-uniform hypergraph $F^*$ with girth larger than $r + 1$, and the property that for
every $s$-set $S$ with $t^{1 - \delta}\le s \le t - 1$, the number of edges in $F^*$ with exactly $r-1$ vertices in $S$ is at least 
$$\frac{1}{10} {s \choose r-1} (t-s) t^{1-r+\frac{1}{2r}}>
 {s \choose r-1} (t-s) t^{1-r+\frac{1}{3r}}=: q_s.$$

 Inside each hyperedge $e$ of $F^*$, place randomly 
a copy of  $G^{*}$. More precisely, among all $r!$ ways to map the vertices of $G^{*}$ to $e$, we pick one with probability $1/r!$. Let $F$ be the resulting graph with $V(F)=V(F^*)$ and $E(F)$ comprises the graph edges in all copies of $G^*$ that lie in edges of $F^*$. 

As $G$ has $r+1$ vertices, and $F^*$ is $r$-uniform, there is no copy of $G$ in $F$ that lies entirely within an edge of $F^*$. If a copy of $G$ in $F$ has two vertices $x,y$ that do not lie in the same edge of $F^*$, then, since $G$ is 2-connected, there is a cycle in $G$ containing $x$ and $y$ and this cycle yields a hypergraph cycle in $F^*$ of length at most $r+1$ which does not exist by construction. We conclude that $F$ is $G$-free.

Furthermore, we claim that for any $s$-set $S$ in $F$, with
$t^{1 - \delta}\le s \le t - 1$,
there exists an edge 
$e$ of $F^*$ with $r-1$ vertices in $S$ and one vertex in $V(F^*) - S$ such that
\begin{equation} \label{specialprop}
\hbox{ \parbox{5in}{\sl \begin{center}
 the copy of $G^{*}$ placed inside $e$ 
induces a copy of $G^{**}$ within $e \cap S$.
 \end{center}}}\end{equation}
 Indeed, (\ref{specialprop}) follows from the following argument. For each  of the $q_s$ edges $e$ of $F^*$ with exactly one vertex outside $S$, the probability that $e$ fails  (\ref{specialprop}) is at most $1-1/r!$. Since any two such edges $e,e'$ share at most one vertex by the girth property of $F^*$,  the probability that all of these $q_s$ edges   $e$ fail (\ref{specialprop}) is at most $(1-1/r!)^{q_s}$. Consequently, the probability that there exists an $s$-set for which there is no $e$ satisfying 
 (\ref{specialprop}) is at most
$$\sum_{s=t^{1-\delta}}^{t-1} {t \choose s} e^{-q_s/r!}
=\sum_{s=t^{1-\delta}}^{t-1} {t \choose t-s} e^{-q_s/r!}
< t{t\choose t-s}e^{-q_s/r!}<e^{\log t +(t-s)\log t - q_s/r!}<1.$$
The final inequality holds as $\delta=1/5r^2$ implies
$$1-r +\frac{1}{3r} +(1-\delta)(r-1) >0.$$
Hence we may assume that for all $s$-sets $S$ with $t^{1 - \delta}\le s \le t - 1$ there exists 
an edge 
$e$ of $F^*$ with $r-1$ vertices in $S$ and one vertex in $V(F^*) - S$
which satisfies (\ref{specialprop}). 

Now let $H$ be any $G$-free $n$-vertex graph. We are to find an $F$-free set of size $\Omega(n^{1 - \varepsilon})$. Let $R = r! + 1$ and 
\[ T = t! \left(R{t - 1 \choose r-1} - 1 \right)^{t}.\]
Set $b=t^{1-\delta}$. We claim the number of copies of $F$ in $H$ is at most
\begin{equation} \label{Tb}  T{n \choose b}.\end{equation}

If (\ref{Tb}) holds, then  there are at most  $O(n^b)$ copies of $F$ in $H$ and we finish the proof as follows. Consider the $t$-uniform hypergraph $\mathcal{H}$ with $V(\mathcal{H}) = V(H)$ and $E(\mathcal{H}) = \{V(F) : F \subseteq H\}$. The average degree of $\mathcal{H}$ is $O(n^{b-1})$.
By Lemma~\ref{spencer}, $\mathcal{H}$ contains an independent set $I$ of size $\Omega(n^{1 - (b-1)/(t-1)})$.
Since $(b-1)/(t-1)<t^{1-\delta}/(t-1)< \varepsilon$ for large $t$, we conclude that $I$ is an $F$-free set of size at least $\Omega(n^{1 - \varepsilon})$.

We now prove (\ref{Tb}). Assume to the contrary. Then to each copy of $F$ we may associate any $b$-subset of its vertices. By pigeonhole, there exists a set $C$ of $b$ vertices in $H$ and at least $T$ copies of $F$, say $F_1,F_2,\dots,F_T$ for which $V(F_i) \cap V(F_j) \supseteq C$. Amongst these sets of size $t$, Lemma~\ref{ersun} gives a sunflower of size $R{t - 1 \choose r-1}$ with core $S \supseteq C$. As $t^{1-\delta}\le |S| \le t-1$, by (\ref{specialprop}), for each of these $R{t - 1 \choose r-1}$ copies $A$ of $F$, there is a vertex $v_A$ outside $S$ that forms an edge $e_A$ in $A$ with $r-1$ vertices in $S$ and $v_A$ plays the role of  vertex $v$ in the copy of $G^*$ within $e_A$ (in other words, $e'_A=e_A-\{v_A\}$ induces a copy of $G^{**}$). 
By pigeonhole,  there exists a set $e' \subseteq S$ of size $r-1$
and vertices $v_1,v_2, \ldots, v_R \not \in S$   such that $e_i=e' \cup \{v_i\}$ is an edge  of $F^*$ for all $i\in [R]$ and $v_i$ plays the role of $v$ in the copy of $G^{*}$ in $e_i$ (in other words, $e'$ induces a copy of $G^{**}$). Since $R > r!$, we can find vertices, say $v_1,v_2$, such that the copies of $G^{**}$ within $e'$ for both $v_1$ and $v_2$  are identical. This copy of $G^{**}$ together with $v_1$ and $v_2$ is a copy of $G^+$. We conclude $H \supseteq G^+ \supseteq G$, a contradiction. 
 \qed

\iffalse {\bf It might be that random graphs work more easily if $G$ is sparse. Indeed, let $G$ be a sparse $k$-chromatic graph and let $F$ be complete 
$(k - 1)$-partite with parts of size $t$. Then a random $G$-free graph can have density $n^{-(v(G) - 2)/(e(G) - 1)}$ will be $F$-free if 
$(v(F) - 2)/(e(F) - 1) > (v(G) - 2)/(e(G) - 1)$.}
\fi

\section{Appendix}

\begin{aproposition}\label{efr-theorem} {\bf (Erd\H{o}s-Frankl-R\"{o}dl)}
For any $N, R \geq 3$ such that $N \geq R \geq \log N$, there exists a linear triangle-free $N$-vertex $R$-uniform hypergraph $H$ with
\[ |E(H)| \geq \frac{N^2}{R^{8\sqrt{\log_R N}}}.\]
\end{aproposition}

\begin{proof} The construction is based on the construction of Behrend~\cite{Behrend} of a dense subset of $\{1,2,\dots,n\}$ with no three-term arithmetic
progression. For completeness, we describe this construction here, which is slightly 
better than the construction of Erd\H{o}s, Frankl and R\"{o}dl~\cite{EFR}. Let $A$ be the set of positive integer points on the sphere of radius $r$ in $\mathbb R^d$.
For any choice of positive integers $x_1,x_2,\dots,x_{d-4} \leq r/\sqrt{d}$, there exist positive integers $x_{d-3},x_{d-2},x_{d-1},x_d$ such that
$x_1^2 + x_2^2 + \dots + x_d^2=r^2$ by Lagrange's four squares theorem. Therefore
\[ |A| \geq \Bigl(\frac{r}{\sqrt{d}}\Bigr)^{d-4}.\]
Let $X_i = [ir]^d$. Then define an $R$-uniform $R$-partite hypergraph $H$ where $V(H)$ consists of $X_1 \cup X_2 \cup \dots \cup X_R$ and let $E(H) = \{x,x+a,x+2a,\dots,x+(R - 1)a\}$ where
$a \in A$ and $x \in X_1$. Then
\[ |V(H)| = N  \leq R^{d + 1}r^d \quad \quad \mbox{ and } \quad \quad |E(H)| = |A||X_1| \ge  \frac{r^{2d-4}}{\sqrt{d}^{d - 4}}.\]
Put $d = \lfloor \sqrt{\log_R N}\rfloor < \log N \leq R$ and $r = R^d $. Then $r^4\le R^{4d}$ and $d^d< R^d$ and hence 
\[ |E(H)| \geq
\frac{N^2}{R^{2d+2}r^4d^{\frac{d-4}{2}}}\ge \frac{N^2}{R^{8d}}>
\frac{N^2}{R^{8\sqrt{\log_R N}}}.\]
This establishes (i). If $e = \{x,x+a,x+2a,\dots,x+(R - 1)a\}$ and $f = \{y,y+b,y+2b,\dots,y+(R - 1)b\}$ intersect in two
vertices of $H$, say $x + ia = y + ib$ and $x + ja = y + jb$, then $x = y$ and $a = b$, establishing (ii). If $e,f$ and
$g = \{z,z+c,z+2c,\dots,z+(R - 1)c\}$ have $|e \cap f| = |f \cap g| = |g \cap e| = 1$, then we may assume
$x + ia = y + ib$ and $y + jb = z + jc$ and $z + kc = x + ka$ for some distinct $i,j,k \in \{0,1,2,\dots,R - 1\}$ and
$a,b,c \in A$. This implies $i(b - a) + j(c - b) + k(a - c) = 0$ which means $(k - i)a + (i - j)b + (j - k)c = 0$.
Since the sphere is strictly convex, $a,b, c$ cannot all lie in a line, and hence we conclude two of $i,j,k$ are identical, a contradiction. This proves (iii).
\end{proof}

For the next proposition, we need some definitions. A cycle of length two in a hypergraph is a set of two edges that share at least two vertices. A cycle of length $\ell>2$ is a collection of $\ell$ distinct vertices $v_1, v_2, \ldots, v_{\ell}$ and $\ell$ distinct edges $e_1, \ldots, e_{\ell}$ where $e_i \cap e_{i+1} = \{v_{i+1}\}$ (indices modulo $\ell$) and $e_i \cap e_j=\emptyset$ otherwise. So an $\ell$-cycle in an $r$-uniform hypergraph ($\ell>2$) has $\ell$ edges and $\ell(r-1)$ vertices (these are often called loose cycles). Say that a hypergraph $H$ has girth  $g$ if the length of the shortest cycle in $H$ is $g$.

\begin{aproposition}\label{random} 
Fix  $r \ge 2$ and $\delta =1/5r^2$. For $t$ sufficiently large, there exists a $t$-vertex $r$-uniform hypergraph $F^*$ with girth at least $r+2$ such that for every $s$-subset $S$ with $t^{1-\delta} < s < t$,  the number of edges with exactly one vertex outside $S$ is at least 
\begin{equation} \label{Sprop}\frac{1}{10} {s \choose r-1} (t-s) t^{1-r+\frac{1}{2r}}.\end{equation}
\end{aproposition}
 
 \begin{proof}Consider the binomial random $r$-uniform hypergraph $H \sim H^{(r)}(t, p)$ with $t$ vertices where each edge appears independently with probability $p=t^{1-r+\frac{1}{2r}}$. For each $2\le \ell \le r+1$, Let $C_{\ell}$ denote the cycle of length $\ell$ (this is unique except for $\ell=2$) and let $\mathcal{B_{\ell}}$ denote a maximal collection of edge-disjoint copies of $C_{\ell}$ in $H$. Form $F^*$ by starting with $H$ and deleting all $\ell$ edges from every copy of $C_{\ell}$ in $\mathcal{B_{\ell}}$ for all $2 \le \ell \le r+1$. Then, by the maximality of $\mathcal{B_{\ell}}$, the remaining hypergraph $F^*$ has girth at least $r+2$. We will now show that with high probability $F^*$ has the required property.  

    Pick $S  \subset V(F^*)$ of size $s$ where $t^{1-\delta} < s < t$. Call an edge in $H$ with exactly one vertex outside $S$ an $S$-edge.
    Let $X=X_{S}$ be the number of $S$-edges, let $Y_{\ell}=Y_{S, \ell}$ be the number of copies of $C_{\ell}$ that contain at least one $S$-edge and let $Z_{\ell}=Z_{S, \ell}$ be the maximal number of pairwise edge-disjoint copies of $C_{\ell}$, each containing at least one $S$-edge. Obviously, $Z_{\ell} \le Y_{\ell}$. Define the event $$A_{\ell}=A_{S, \ell}= \{X> 10 r \ell \, Z_{\ell}\}.$$ We note that if $A_{S,\ell}$ holds for every appropriate $S$, and $2\le \ell \le r+1$,  then the number of $S$-edges in $F^*$ is at least 
    $$X- \sum_{\ell=2}^{r+1} \ell \, Z_{\ell} \ge |X| - \sum_{\ell=2}^{r+1} \frac{|X|}{10r}>(0.9) |X|.$$ 
    Moreover, $\EE(X) = {s \choose r-1} (t-s)p$, so if it is also the case that $X > \EE(X)/2$, then the number of $S$-edges in $F^*$ is at least $(0.4)\EE(X)$ and $S$ satisfies (\ref{Sprop}).

We see that  
$$ \EE(Y_{\ell}) <  {s \choose r-1}(t-s) t^{\ell(r-1)-r} p^{\ell}.$$
As $p=t^{1-r+\frac{1}{2r}}$ and $\ell\le r+1$, we have $p^{\ell}t^{\ell(r-1)-r} \ll p$. Therefore $\EE(Y_{\ell}) \ll \EE(X)$. Now $$\PP(\overline{A_{\ell}}) =\PP(X \le 10r\ell \, Z_{\ell}) \le
\PP\left(X \le  \frac{\EE(X)}{2}\right) +\PP\left(Z_{\ell} \ge \frac{\EE(X)}{20r\ell}\right). 
$$
 Krivelevich~\cite[Claim 1]{Kr} proved that in this setup, for any constant $c>0$,
$$\PP(Z_{\ell}\ge c \, \EE(Y_{\ell})) < e^{-c\,(\log c-1)\EE(Y_{\ell})}.$$
Using this and $\EE(Y_{\ell}) \ll \EE(X)$ we have
$$\PP\left(Z_{\ell} \ge \frac{\EE(X)}{20 r \ell}\right) 
=\PP\left(Z_{\ell} \ge \frac{\EE(X)}{20 r \ell \, \EE(Y_{\ell})} \EE(Y_{\ell})\right) 
<e^{\frac{-\EE(X)}{20r\ell} (\log(\frac{\EE(X)}{20r\ell \,\EE(Y_{\ell})})-1)} < e^{-\EE(X)}.
$$
The standard Chernoff bound gives 
$\PP(X \le \EE(X)/2) < e^{-\EE(X)/8}$ so altogether we obtain $\PP(\overline{A_{\ell}})< e^{-\EE(X)/9}$. Using the union bound, the probability that there exists an $S$ that fails (\ref{Sprop}) is at most
$$\sum_{s=t^{1-\delta}}^{t-1} {t \choose s} e^{-{s \choose r-1}(t-s)p/9} < e^{ \log t +(t-s)\log t - {s \choose r-1}(t-s)p/9}.$$
The power of $t$ in $s^{r-1} p$ is at least $1-r+1/2r+(1-\delta)(r-1)>0$ as $\delta < 1/2r^2$ and hence the quantity above vanishes for large $t$. We conclude that  (\ref{Sprop}) holds in $F^*$ with high probability. 
 \end{proof}

\end{document}